\documentclass[a4paper,11pt]{article}

\usepackage[T1]{fontenc}
\usepackage[utf8]{inputenc}
\usepackage[provide=*,english]{babel}
\usepackage{amsmath,amssymb,amsthm,mathtools,parskip,csquotes}
\usepackage[margin=2.7cm]{geometry}

\theoremstyle{plain}
\newtheorem{theorem}{Theorem}[section]
\newtheorem{lemma}[theorem]{Lemma}
\newtheorem{proposition}[theorem]{Proposition}
\newtheorem{corollary}[theorem]{Corollary}
\newtheorem*{conjecture}{Conjecture}
\newtheorem*{mainresult}{Main Result}

\theoremstyle{definition}
\newtheorem{definition}[theorem]{Definition}

\usepackage{titlesec}

\titleformat{\section}
  {\normalfont\large\bfseries}
  {\thesection}
  {1em}
  {}
\titleformat{\subsection}
  {\normalfont\normalsize\bfseries}
  {\thesubsection}
  {1em}
  {}

\usepackage[
    backend=biber,
    style=alphabetic,
    giveninits=true,
    sorting=nyt,
    maxbibnames=99
]{biblatex}

\renewbibmacro*{in:}{%
    \ifentrytype{article}
        {}
        {\printtext{\bibstring{in}\intitlepunct}}%
}

\DeclareFieldFormat[article]{number}{no\adddot\space#1}

\renewbibmacro*{volume+number+eid}{%
    \printfield{volume}%
    \setunit{\addcomma\space}%
    \printfield{number}%
    \setunit{\addcomma\space}%
    \printfield{eid}%
}

\usepackage[hidelinks]{hyperref}

\newcommand{\diff}{\,\mathrm d}

\title{The Petty Conjecture for Convex Bodies of Revolution}
\author{Florian Mielke-Sulz\thanks{Email: \texttt{florian.mielkesulz@gmail.com}}}
\begin{document}

\maketitle

\begin{abstract}
\noindent
The Petty conjecture asserts that the quotient \(V_n(\Pi K)/V_n(K)^{n-1}\) is minimized among all convex bodies precisely by ellipsoids. In this paper, we prove the conjecture for all convex bodies of revolution in \(\mathbb R^n\), \(n\geq3\), including the equality cases.
\end{abstract}

{\small
\par
}

\section{Introduction}
Here and throughout, a convex body in \(\mathbb R^n\) is a compact convex set with nonempty interior. For \(m\geq1\), we denote by \(V_m\) the \(m\)-dimensional Lebesgue measure, by \(B^m\) the Euclidean unit ball in \(\mathbb R^m\), by \(\kappa_m=V_m(B^m)\) its volume, and by \(\mathbb S^{m-1}=\partial B^m\) the Euclidean unit sphere. Moreover, \(\langle\cdot,\cdot\rangle\) denotes the Euclidean inner product. The support function of a convex body \(K\) is defined by
\[
    h_K(u)=\max_{x\in K}\langle x,u\rangle,
    \qquad u\in\mathbb S^{n-1},
\]
and its projection body \(\Pi K\) is the origin-symmetric convex body whose support function is given by
\[
    h_{\Pi K}(u)
    =
    V_{n-1}\bigl(\operatorname{proj}_{u^\perp}K\bigr),
    \qquad u\in\mathbb S^{n-1},
\]
where \(\operatorname{proj}_{u^\perp}K\) denotes the orthogonal projection of \(K\) onto the hyperplane \(u^\perp\).

The conjecture formulated by Petty \cite{petty1971isoperimetricproblems} has been an open problem in convex geometry for decades \cite{gardner2025geometrictomographyupdate}. For further background, see also \cite{lutwak1990conjecturedprojection}.

\begin{conjecture}[Petty conjecture]
Let \(n\geq3\), let \(K\subseteq\mathbb R^n\) be a convex body, and let \(B^n\subseteq\mathbb R^n\) be the Euclidean unit ball. Then
\begin{equation} \label{conj}
    \frac{V_n(\Pi K)}{V_n(K)^{n-1}}
    \geq
    \frac{V_n(\Pi B^n)}{V_n(B^n)^{n-1}}.
\end{equation}
Equality holds if and only if \(K\) is an ellipsoid.
\end{conjecture}

In a recent preprint, Chen, Feng, Li, Xi, and Xu prove the Petty
conjecture in dimension three, including the equality cases
\cite{ChenFengLiXiXu2026}. For general convex bodies in dimensions \(n\geq4\), the conjecture remains open. Earlier partial results include results for certain classes of three-dimensional bodies \cite{saroglou2011volumesprojectionbodies}, a local result for convex bodies with absolutely continuous surface area measure whose curvature function, after a linear transformation, is sufficiently close to the constant function \(1\) in \(L^\infty\) \cite{SaroglouZvavitch2017}, and a result under additional assumptions on the minimal surface area within the volume-preserving affine class \cite{GiannopoulosPapadimitrakis1999}. Saroglou also showed that a classical proof approach based on Steiner symmetrization cannot work, since the Petty quotient \(V_n(\Pi K)/V_n(K)^{n-1}\) is not monotone under Steiner symmetrization in general \cite{saroglou2011volumesprojectionbodies}.

The class reduction due to Schneider \cite{schneider1987geometricinequalities} and extended by Lutwak \cite{Lutwak1990Quermassintegrals} to projection bodies of different degrees shows that, after a suitable translation, every minimizer of the Petty quotient satisfies
\[
    \Pi^2K=\alpha K
    \qquad\text{for some }\alpha>0,
\]
where \(\Pi^2K=\Pi(\Pi K)\). However, Weil showed that solutions of this equation need not be ellipsoids \cite{Weil1971}. On the other hand, Ivaki proved that among the solutions of this equation for which \(h_{AK}\) is sufficiently close to \(1\) in \(C^2(\mathbb S^{n-1})\) for some invertible linear map \(A\), only origin-centered ellipsoids occur \cite{ivaki2018localuniqueness}. This local uniqueness result was subsequently generalized by Ortega-Moreno and Schuster to a broad class of sufficiently regular Minkowski valuations, further developing the fixed-point approach to the Petty conjecture \cite{ortegamoreno2021fixedpointsminkowskivaluations}.

A convex body \(K\subseteq\mathbb R^n\) is called a body of revolution if it is invariant under all rotations fixing some affine line pointwise. Since translations and orthogonal transformations preserve both the Petty quotient and the class of bodies of revolution, we may assume that this axis is \(\mathbb Re_n\), where \(e_n=(0,\ldots,0,1)\) denotes the \(n\)-th standard basis vector. For bodies of revolution, Saroglou proved a non-sharp lower bound \cite{saroglou2015secondprojectionbody}. In dimensions \(n\geq4\), the sharp inequality together with its equality cases had remained open for this class. We prove the Petty conjecture, including its equality cases, for convex bodies of revolution in every dimension \(n\geq3\):

\begin{mainresult}[Petty conjecture for convex bodies of revolution]
Let \(n\geq3\), and let \(K\subseteq\mathbb R^n\) be a convex body of revolution. Then \eqref{conj} holds, with equality if and only if \(K\) is an ellipsoid.
\end{mainresult}

Independently, Dorrek found a proof of the same result with the assistance
of AI and describes on his website how he arrived at it \cite{Dorrek2026}.

The proof presented in this paper begins with normalized Blaschke symmetrization, which
preserves the projection body. By the Kneser--S\"uss inequality,
this operation does not decrease volume and hence does not increase
the Petty quotient. It therefore suffices to consider
origin-symmetric bodies. After a suitable linear normalization, a
smooth body of revolution \(K=-K\) has the representation
\[
    K=\{(x,z):x\in B^{n-1},\ |z|\leq f(|x|)\},
\]
where \(f\) is a concave boundary profile. The corresponding profile of
the unit ball is \(f_0(r)=\sqrt{1-r^2}\). With \(y=r^{n-1}\), set
\[
    b_K(y)=-f'(y^{1/(n-1)}),
    \qquad
    b_0(y)=-f_0'(y^{1/(n-1)}).
\]
Writing \(g=b_K/b_0\) and \(b_g=gb_0=b_K\), we introduce integral
functionals \(I\) and \(T\) satisfying
\[
    V_n(K)=\frac{2\kappa_{n-1}}{n-1}I(b_g),
    \qquad
    V_n(\Pi K)=\frac{2^n\kappa_{n-1}^{\,n}}{n!}T(b_g),
\]
where the first identity follows from integration by parts and
the second from the volume formula for zonoids. Comparison with the unit ball reduces \eqref{conj} to the analytic
inequality
\[
    \frac{T(b_g)}{T(b_0)}
    \geq
    \left(\frac{I(b_g)}{I(b_0)}\right)^{n-1}.
\]
A determinant estimate reduces this inequality to a lower bound for a
multilinear functional \(\mathcal B\). The spherical
Blaschke--Petkantschin formula expresses its diagonal in terms of
determinants of moment matrices. Rotational symmetry reduces these
determinants to the axial and transverse eigenvalues. Integration by
parts then yields a weighted integral of \(\alpha_g^{n-1}\), where
\(\alpha_g\) is a scalar integral transform of \(g\). Hölder's inequality
gives the required bound.

This proves the inequality for smooth bodies and, by approximation,
for arbitrary bodies of revolution. For the nonsmooth equality case,
we introduce profile measures \(\rho_K\) and \(\rho_0\), which also
capture cylindrical boundary parts. The equality condition in Hölder's
inequality and the injectivity of the extended transform imply
\(\rho_K=c\rho_0\), which characterizes the ellipsoidal profile. Finally, equality in the Kneser--S\"uss inequality shows that
the original body is a translate of its Blaschke symmetral,
which completes the characterization of equality.

\section{Notation and Preliminaries}
Throughout, let \(n\geq3\). For \(m\geq1\), let \(\diff\sigma_m\) denote the surface measure on \(\mathbb S^m\). Moreover, let
\[
    \diff\overline\sigma_m
    =
    \frac{\diff\sigma_m}{(m+1)\kappa_{m+1}}
\]
be the associated rotation-invariant probability measure. By \(\operatorname{proj}_E\) we denote the orthogonal projection onto a linear subspace \(E\), by \(u^\perp\) the orthogonal complement of \(u\in\mathbb R^n\), by \(\mathcal H^m\) the \(m\)-dimensional Hausdorff measure, by \(G(n,k)\) the Grassmannian of \(k\)-dimensional linear subspaces of \(\mathbb R^n\), and by \(\operatorname{GL}(n)\) the group of invertible linear transformations on \(\mathbb R^n\). For \(A\in\operatorname{GL}(n)\), we set \(A^{-T}=(A^{-1})^T\). 

A zonotope is a finite Minkowski sum of line segments.
A convex body is called a zonoid if it is a limit of
zonotopes in the Hausdorff metric. Every projection body is an origin-symmetric zonoid. 

We denote the surface area measure of a convex body \(K\) by \(S_K\). By the existence and uniqueness theorem for the Minkowski problem, the Blaschke sum \(K\#L\) is determined up to translation by \(S_{K\#L}=S_K+S_L\). We denote by \(K^{\mathrm B}\) the origin-symmetric representative of
the normalized Blaschke sum \(2^{-1/(n-1)}(K\#(-K))\), that is,
\begin{equation}\label{eq:blaschke-symmetral}
    K^{\mathrm B}
    =
    2^{-1/(n-1)}(K\#(-K)).
\end{equation}
Finally, \(C^\infty_+\) denotes the class of convex bodies \(K\) whose boundary \(\partial K\) is a \(C^\infty\) hypersurface with everywhere positive Gaussian curvature. Bodies in this class will be called smooth. For further background on these notions, we refer the reader to Schneider \cite{Schneider2014}.

We now introduce the analytic quantities to which the geometric problem will later be reduced. For \(g\in C([0,1])\) with \(g\geq0\) and \(0\leq y<1\), set
\[
    b_0(y)
    =
    \frac{y^{1/(n-1)}}{\sqrt{1-y^{2/(n-1)}}},
    \qquad
    b_g(y)=g(y)b_0(y),
    \qquad
    I(b_g)=\int_0^1 b_g(y)y^{1/(n-1)}\,\diff y
\]
and
\[
\begin{aligned}
    T(b_g)
    =\int_{(0,1)^n}\int_{(\mathbb S^{n-2})^n}
      &\left|
      \det\bigl(
          (b_g(y_1)\theta_1,1),\ldots,
          (b_g(y_n)\theta_n,1)
      \bigr)
      \right|\\
      &\diff\overline\sigma_{n-2}(\theta_1)\cdots
       \diff\overline\sigma_{n-2}(\theta_n)\,
       \diff y_1\cdots\diff y_n.
\end{aligned}
\]
The values of \(b_0\) and \(b_g\) at \(y=1\) may be chosen arbitrarily, since these functions are used only up to equality almost everywhere.
Since \(b_0\in L^1([0,1])\) and the determinant is a sum of products of \(n-1\) values \(b_g(y_i)\), the quantities \(I(b_g)\) and \(T(b_g)\) are finite.

For \(u\in\mathbb S^{n-1}\), let \(\eta(u)=(1-\langle u,e_n\rangle^2)^{(n-1)/2}\), and for \(u_1,\ldots,u_{n-1}\in\mathbb S^{n-1}\), set
\[
    \Delta(u_1,\ldots,u_{n-1})
    =
    \sqrt{\det\bigl(\langle u_i,u_j\rangle\bigr)_{i,j=1}^{n-1}}.
\]
Outside the null set of linearly dependent tuples, there exists a unique oriented unit normal vector
\(N=N(u_1,\ldots,u_{n-1})\) (see, e.g.,
\cite[Section~1.7]{Federer1969}) such that
\begin{equation} \label{unitvek}
    \det(u_1,\ldots,u_{n-1},v)
    =
    \Delta(u_1,\ldots,u_{n-1})\langle N,v\rangle,
    \qquad v\in\mathbb R^n.
\end{equation}

For \(g_1,\ldots,g_{n-1}\in C([0,1])\), define
\begin{equation} \label{defB}
\begin{aligned}
    \mathcal B(g_1,\ldots,g_{n-1})
    =
    \frac1{(2\kappa_{n-1})^{n-1}}
    \int_{(\mathbb S^{n-1})^{n-1}}
      &\Delta(u_1,\ldots,u_{n-1})
       \langle N,e_n\rangle^2\\
      &g_1(\eta(u_1))\cdots g_{n-1}(\eta(u_{n-1}))\\
      &\diff\sigma_{n-1}(u_1)\cdots
       \diff\sigma_{n-1}(u_{n-1}),
\end{aligned}
\end{equation}
where the kernel is set equal to zero on the set of linearly dependent tuples.

\section{The Analytic Inequality}
The analytic core of the proof is the comparison of the functionals \(T\) and \(I\) established in Lemma~\ref{lem:analytisch}. It is based on the determinant estimate in Lemma~\ref{lem:determinantenvergleich}, the factorization of \(\mathcal B\) in Lemma~\ref{lem:B-faktorisierung}, and Hölder's inequality.

\begin{lemma}[Determinant estimate]\label{lem:determinantenvergleich}
Let \(g\in C([0,1])\) be nonnegative. Then
\begin{equation*}
    T(b_g)
    \geq
    n\,\mathcal B(g,\ldots,g).
\end{equation*}
If \(g\) is constant, then equality holds. In particular,
\begin{equation*}
    T(b_0)
    =
    n\,\mathcal B(1,\ldots,1).
\end{equation*}
\end{lemma}

\begin{proof}
For \(y\in(0,1)\) and \(\theta\in\mathbb S^{n-2}\), set
\[
    u=u(y,\theta)
    =
    \left(
        y^{1/(n-1)}\theta,
        \sqrt{1-y^{2/(n-1)}}
    \right)
    \in\mathbb S^{n-1}_+,
\]
where \(\mathbb S^{n-1}_+=\{u\in\mathbb S^{n-1}:\langle u,e_n\rangle>0\}\). Under the change of variables 
\[
(y,\theta)\mapsto u(y,\theta), \quad (0,1)\times\mathbb S^{n-2}\to\mathbb S^{n-1}_+\setminus\{e_n\},
\]
the change-of-variables formula reads
\begin{equation}\label{eq:direkter-sphaerenwechsel}
    \diff\overline\sigma_{n-2}(\theta)\,\diff y
    =
    \frac{\langle u,e_n\rangle}{\kappa_{n-1}}\,
    \diff\sigma_{n-1}(u).
\end{equation}
For this parametrization, \(\eta(u)=y\). Writing \(u=(\overline u,\langle u,e_n\rangle)\) with \(\overline u\in\mathbb R^{n-1}\), and setting \(G(u)=g(\eta(u))\), we have
\begin{equation}\label{eq:bg-spherical-vector}
    (b_g(y)\theta,1)
    =
    \frac1{\langle u,e_n\rangle}
    \bigl(G(u)\overline u,\langle u,e_n\rangle\bigr).
\end{equation}
Now write \(u_i=u(y_i,\theta_i)\) for \(i=1,\ldots,n\). The factors \(\langle u_i,e_n\rangle^{-1}\) and
\(\langle u_i,e_n\rangle\) in \eqref{eq:bg-spherical-vector} and
\eqref{eq:direkter-sphaerenwechsel}, respectively, cancel for each
\(i\). Hence, using also that
\[
    \left|
    \det\bigl(
        (G(u_1)\overline u_1,\langle u_1,e_n\rangle),\ldots,
        (G(u_n)\overline u_n,\langle u_n,e_n\rangle)
    \bigr)
    \right|
\]
is even in each variable \(u_i\), we obtain
\begin{equation}\label{eq:T-sphaerisch}
    T(b_g)
    =
    \frac1{(2\kappa_{n-1})^n}
    \int_{(\mathbb S^{n-1})^n}
    |D_g(u_1,\ldots,u_n)|
    \,\diff\sigma_{n-1}(u_1)\cdots\diff\sigma_{n-1}(u_n),
\end{equation}
where
\[
    D_g(u_1,\ldots,u_n)
    =
    \det\bigl(
        (G(u_1)\overline u_1,\langle u_1,e_n\rangle),\ldots,
        (G(u_n)\overline u_n,\langle u_n,e_n\rangle)
    \bigr).
\]

Set \(D_1=\det(u_1,\ldots,u_n)\) and \(\varepsilon=\operatorname{sgn}D_1\). Pointwise,
\[
    |D_g|\geq \varepsilon D_g.
\]
Hence, \eqref{eq:T-sphaerisch} gives
\begin{equation}\label{eq:T-untere-schranke}
    T(b_g)
    \geq
    \frac1{(2\kappa_{n-1})^n}
    \int_{(\mathbb S^{n-1})^n}
    \varepsilon D_g
    \,\diff\sigma_{n-1}(u_1)\cdots\diff\sigma_{n-1}(u_n).
\end{equation}

We now compute the right-hand side. Expanding \(D_g\) along the last coordinates gives
\begin{equation*}
    D_g
    =
    \sum_{i=1}^n
    (-1)^{i+n}\langle u_i,e_n\rangle
    \det(\overline u_1,\ldots,\widehat{\overline u_i},\ldots,\overline u_n)
    \prod_{j\neq i}G(u_j),
\end{equation*}
where \(\widehat{\cdot}\) means that the corresponding entry is omitted. For the \(i\)-th summand, set
\(
    W_i=(u_1,\ldots,\widehat{u_i},\ldots,u_n).
\)
Outside the null set of linearly dependent tuples, let \(N_i\) denote the oriented unit normal vector associated with \(W_i\) as in \eqref{unitvek}. Since
\[
    D_1=(-1)^{n-i}\det(W_i,u_i),
\]
we have
\[
    \varepsilon
    =
    (-1)^{n-i}\operatorname{sgn}\det(W_i,u_i).
\]
Thus, the two powers of \((-1)\) appearing in the \(i\)-th summand of \(\varepsilon D_g\) cancel, since \((-1)^{i+n}(-1)^{n-i}=1\).

We use the identities
\begin{equation}\label{eq:sphaerische-sign-identitaet}
    \int_{\mathbb S^{n-1}}
    \operatorname{sgn}\det(u_1,\ldots,u_{n-1},v)\,\langle v,e_n\rangle
    \,\diff\sigma_{n-1}(v)
    =
    2\kappa_{n-1}\langle N,e_n\rangle
\end{equation}
and
\begin{equation}\label{eq:kofaktor-N}
    \det(\overline u_1,\ldots,\overline u_{n-1})
    =
    \Delta(u_1,\ldots,u_{n-1})
    \langle N,e_n\rangle.
\end{equation}
The first follows from
\[
    \int_{\mathbb S^{n-1}}
    \operatorname{sgn}\langle N,v\rangle\,v
    \,\diff\sigma_{n-1}(v)
    =
    2\kappa_{n-1}N,
\]
and the second follows immediately from the definition of \(N\) with \(v=e_n\).

For fixed \(u_j\), \(j\neq i\), the \(i\)-th summand of
\(\varepsilon D_g\) is
\[
    \operatorname{sgn}\det(W_i,u_i)\,
    \langle u_i,e_n\rangle\,
    \det(\overline u_1,\ldots,\widehat{\overline u_i},\ldots,\overline u_n)
    \prod_{j\neq i}G(u_j).
\]
Applying \eqref{eq:sphaerische-sign-identitaet} and
\eqref{eq:kofaktor-N} to the ordered tuple \(W_i\), integration of the
\(i\)-th summand of \(\varepsilon D_g\) with respect to \(u_i\) yields
\[
    2\kappa_{n-1}
    \Delta(u_1,\ldots,\widehat{u_i},\ldots,u_n)
    \langle N_i,e_n\rangle^2
    \prod_{j\neq i}G(u_j).
\]
Together with the prefactor in \eqref{eq:T-untere-schranke}, the contribution
of the \(i\)-th summand is therefore precisely
\(
    \mathcal B(g,\ldots,g).
\)
Since all \(n\) summands give the same contribution, it follows that
\[
    \frac1{(2\kappa_{n-1})^n}
    \int_{(\mathbb S^{n-1})^n}
    \varepsilon D_g
    \,\diff\sigma_{n-1}(u_1)\cdots\diff\sigma_{n-1}(u_n)
    =
    n\,\mathcal B(g,\ldots,g).
\]
Together with \eqref{eq:T-untere-schranke}, this yields
\[
    T(b_g)\geq n\,\mathcal B(g,\ldots,g).
\]

If \(g\equiv c\) is constant, then \(D_g=c^{n-1}D_1\). Since \(c\geq0\), we thus have \(|D_g|=\varepsilon D_g\) pointwise, and hence equality holds. In particular, for \(c=1\),
\[
    T(b_0)=n\,\mathcal B(1,\ldots,1).
\]
\end{proof}

The classical Blaschke--Petkantschin formulas decompose integrals
over tuples of points according to the linear or affine subspaces
they span \cite[Section~7.2]{SchneiderWeil2008}.
To factorize \(\mathcal B\), we use the following spherical version
\cite[Lemma~3.2]{BaranyHugReitznerSchneider2017}.
The extension from nonnegative to integrable functions follows
by considering the positive and negative parts separately.

\begin{lemma}[Spherical Blaschke--Petkantschin formula]
\label{lem:spherical-BP}
Let \(F\colon(\mathbb S^{n-1})^{n-1}\to\mathbb R\) be integrable. Then
\[
\begin{aligned}
&\int_{(\mathbb S^{n-1})^{n-1}}
F(u_1,\ldots,u_{n-1})
\,\diff\sigma_{n-1}(u_1)\cdots
\diff\sigma_{n-1}(u_{n-1})\\
&\quad =
\frac{n\kappa_n}{2}
\int_{G(n,n-1)}
\int_{(\mathbb S^{n-1}\cap H)^{n-1}}
F(u_1,\ldots,u_{n-1})
\Delta(u_1,\ldots,u_{n-1})\\
&\qquad\qquad
\diff\sigma_H(u_1)\cdots
\diff\sigma_H(u_{n-1})\,\diff H,
\end{aligned}
\]
where \(\diff\sigma_H\) denotes surface measure on
\(\mathbb S^{n-1}\cap H\), and \(\diff H\) denotes the
rotation-invariant probability measure on \(G(n,n-1)\).
\end{lemma}

To state and prove the factorization lemma below, we introduce the following notation. For \(g\in C([0,1])\) and \(0\leq t\leq1\), set
\begin{equation}\label{eq:alpha-definition}
    \alpha_g(t)
    =
    \kappa_{n-2}\int_{-1}^1
    g\bigl((1-t^2s^2)^{(n-1)/2}\bigr)
    (1-s^2)^{(n-2)/2}\,\diff s.
\end{equation}

\begin{lemma}[Factorization of \(\mathcal B\)]
\label{lem:B-faktorisierung}
There exists a constant \(c_n>0\) depending only on the dimension such that, for all \(g_1,\ldots,g_{n-1}\in C([0,1])\),
\begin{equation}\label{eq:B-faktorisierung}
    \mathcal B(g_1,\ldots,g_{n-1})
    =
    c_n
    \int_0^1
    \frac{t^n}{\sqrt{1-t^2}}\,
    \alpha_{g_1}(t)\cdots\alpha_{g_{n-1}}(t)\,\diff t.
\end{equation}
Moreover,
\begin{equation}\label{eq:B-I-quotient}
    \frac{\mathcal B(g,1,\ldots,1)}
         {\mathcal B(1,\ldots,1)}
    =
    \frac{I(b_g)}{I(b_0)}.
\end{equation}
\end{lemma}

\begin{proof}
\emph{Reduction to the Grassmannian.}
In what follows, \(C_n>0\) denotes a constant depending only on the dimension whose value may change from line to line. For a hyperplane \(H\in G(n,n-1)\) with unit normal vector \(N_H\) and \(g\in C([0,1])\), define the linear operator \(M_H(g)\) by
\[
    M_H(g)v
    =
    \int_{\mathbb S^{n-1}\cap H}
    g(\eta(u))\langle u,v\rangle u\,\diff\sigma_H(u),
    \qquad v\in\mathbb R^n,
\]
and let \(\det_H M_H(g)\) denote the determinant of its restriction to \(H\).

Applying Lemma~\ref{lem:spherical-BP} to the integrand in the definition
\eqref{defB} of \(\mathcal B\), we have \(N=\pm N_H\) for
\(u_1,\ldots,u_{n-1}\in\mathbb S^{n-1}\cap H\). Thus,
\(\langle N,e_n\rangle^2=\langle N_H,e_n\rangle^2\), while the factor
\(\Delta(u_1,\ldots,u_{n-1})\) arising from Lemma~\ref{lem:spherical-BP}
multiplies the factor \(\Delta(u_1,\ldots,u_{n-1})\) already present in
\eqref{defB}, yielding \(\Delta(u_1,\ldots,u_{n-1})^2\). Choose an orthonormal basis \(e_1,\ldots,e_{n-1}\) of \(H\). Then
\[
    \Delta(u_1,\ldots,u_{n-1})^2
    =
    \left[
        \det\bigl(
            \langle u_j,e_i\rangle
        \bigr)_{i,j=1}^{n-1}
    \right]^2.
\]
Let \(\mathfrak S_{n-1}\) denote the set of permutations of \(\{1,\ldots,n-1\}\). Expanding both determinants and then applying Fubini's theorem gives
\[
\begin{aligned}
&\int_{(\mathbb S^{n-1}\cap H)^{n-1}}
\Delta(u_1,\ldots,u_{n-1})^2
\prod_{j=1}^{n-1}g(\eta(u_j))
\,\diff\sigma_H(u_1)\cdots\diff\sigma_H(u_{n-1})
\\
&\quad =
\sum_{\pi,\tau\in\mathfrak S_{n-1}}
\operatorname{sgn}(\pi)\operatorname{sgn}(\tau)
\prod_{j=1}^{n-1}
\int_{\mathbb S^{n-1}\cap H}
g(\eta(u))
\langle u,e_{\pi(j)}\rangle
\langle u,e_{\tau(j)}\rangle
\,\diff\sigma_H(u)
\\
&\quad =
(n-1)!
\det\left(
    \langle M_H(g)e_i,e_j\rangle
\right)_{i,j=1}^{n-1}
=
(n-1)!\det_H M_H(g).
\end{aligned}
\]
After adjusting the positive dimensional constant \(C_n\), it follows that
\begin{equation}\label{eq:B-Grassmann}
    \mathcal B(g,\ldots,g)
    =
    C_n\int_{G(n,n-1)}
    \langle N_H,e_n\rangle^2
    \det_H M_H(g)\,\diff H.
\end{equation}

\emph{Eigenvalues of \(M_H(g)\).}
For almost every \(H\in G(n,n-1)\), we have
\(t=|\operatorname{proj}_H e_n|>0\). For such \(H\), set
\(e_H=\frac{\operatorname{proj}_H e_n}{t}\). Choosing \(e_H\) as the first coordinate direction in \(H\) and writing \(s=\langle u,e_H\rangle\), we have, for \(u\in\mathbb S^{n-1}\cap H\),
\[
    \langle u,e_n\rangle=ts,
    \qquad
    \eta(u)=(1-t^2s^2)^{(n-1)/2}.
\]
Since \(g(\eta(u))\) depends on \(u\) only through \(\langle u,e_H\rangle\), the operator \(M_H(g)|_H\) commutes with every orthogonal transformation of \(H\) fixing \(e_H\). Hence, it acts as a scalar on \(H\cap e_H^\perp\). For every unit vector \(w\in H\cap e_H^\perp\),
\[
    \langle M_H(g)w,w\rangle
    =
    \kappa_{n-2}\int_{-1}^1
    h(t,s)(1-s^2)^{(n-2)/2}\,\diff s
    =
    \alpha_g(t),
\]
where \(h(t,s)=g\bigl((1-t^2s^2)^{(n-1)/2}\bigr)\). Thus, \(\alpha_g(t)\) is the eigenvalue on the \((n-2)\) dimensional subspace \(H\cap e_H^\perp\).

If \(g\) is smooth, then the eigenvalue in the direction \(e_H\) is
\[
    \alpha_g(t)+t\alpha_g'(t).
\]
Indeed, \(t\,\partial_t h=s\,\partial_s h\). Integration by parts in the one-dimensional representation of \(\alpha_g\) gives
\[
    t\alpha_g'(t)
    =
    (n-2)\kappa_{n-2}
    \int_{-1}^1
    h(t,s)s^2(1-s^2)^{(n-4)/2}\,\diff s
    -\alpha_g(t),
\]
where the first term on the right is precisely
\(\langle M_H(g)e_H,e_H\rangle\), that is, the eigenvalue of
\(M_H(g)|_H\) in the direction \(e_H\). Hence
\[
    \det_H M_H(g)
    =
    \alpha_g(t)^{n-2}
    \bigl(\alpha_g(t)+t\alpha_g'(t)\bigr).
\]

\emph{Integration by parts and polarization.}
Recall that \(t=|\operatorname{proj}_H e_n|\). Identifying
\(H\in G(n,n-1)\) with \(v^\perp\), where normalized surface measure
on \(\mathbb S^{n-1}\) pushes forward to \(\diff H\), we have
\[
    t=\sqrt{1-\langle v,e_n\rangle^2}.
\]
Hence, for every \(\Phi\in C([0,1])\), spherical coordinates with
\(s=|\langle v,e_n\rangle|\) give
\[
\begin{aligned}
    \int_{G(n,n-1)}\Phi(t)\,\diff H
    &=
    \frac{2(n-1)\kappa_{n-1}}{n\kappa_n}
    \int_0^1
    \Phi\bigl(\sqrt{1-s^2}\bigr)
    (1-s^2)^{(n-3)/2}\,\diff s.
\end{aligned}
\]
After the substitution \(t=\sqrt{1-s^2}\), this becomes
\[
    \int_{G(n,n-1)}\Phi(t)\,\diff H
    =
    C_n\int_0^1
    \Phi(t)\frac{t^{n-2}}{\sqrt{1-t^2}}\,\diff t,
    \qquad
    C_n=\frac{2(n-1)\kappa_{n-1}}{n\kappa_n}>0.
\]
Moreover,
\(\langle N_H,e_n\rangle^2=1-t^2\). Thus, after absorbing the
positive dimensional constant into \(C_n\),
\eqref{eq:B-Grassmann} yields
\[
    \mathcal B(g,\ldots,g)
    =
    C_n\int_0^1
    t^{n-2}\sqrt{1-t^2}\,
    \alpha_g(t)^{n-2}
    \bigl(\alpha_g(t)+t\alpha_g'(t)\bigr)\,\diff t.
\]
Since
\[
    \alpha_g^{n-2}(\alpha_g+t\alpha_g')
    =
    \alpha_g^{n-1}
    +\frac{t}{n-1}(\alpha_g^{n-1})',
\]
integration by parts gives
\[
\begin{aligned}
&\int_0^1
t^{n-2}\sqrt{1-t^2}\,
\alpha_g^{n-2}(\alpha_g+t\alpha_g')\,\diff t\\
&\quad =
\int_0^1
\left[
t^{n-2}\sqrt{1-t^2}
-\frac1{n-1}
\frac{\diff}{\diff t}
\left(t^{n-1}\sqrt{1-t^2}\right)
\right]
\alpha_g(t)^{n-1}\,\diff t\\
&\quad =
\frac1{n-1}
\int_0^1
\frac{t^n}{\sqrt{1-t^2}}\,
\alpha_g(t)^{n-1}\,\diff t.
\end{aligned}
\]
The boundary term vanishes because \(\alpha_g\) is bounded and
\[
    t^{n-1}\sqrt{1-t^2}\longrightarrow0
    \qquad\text{as }t\downarrow0\text{ and }t\uparrow1.
\]
Absorbing the factor \(1/(n-1)\), we obtain, with a now fixed dimensional constant \(c_n>0\),
\[
    \mathcal B(g,\ldots,g)
    =
     c_n
    \int_0^1
    \frac{t^n}{\sqrt{1-t^2}}\,
    \alpha_g(t)^{n-1}\,\diff t.
\]
Both sides of \eqref{eq:B-faktorisierung} define symmetric \((n-1)\)-linear forms in \(g_1,\ldots,g_{n-1}\). Therefore, polarization of the diagonal identity yields \eqref{eq:B-faktorisierung} first for smooth \(g_1,\ldots,g_{n-1}\). Since \(\mathcal B\) is continuous in all its arguments and \(g\mapsto\alpha_g\) is continuous with respect to the supremum norm, uniform approximation proves the formula for all \(g_1,\ldots,g_{n-1}\in C([0,1])\).

\emph{Computation of the mixed term.}
It remains to determine the mixed term. Since \(\alpha_1\) is a positive constant and the integrand in the one-dimensional representation of \(\alpha_g\) is even in \(s\), \eqref{eq:B-faktorisierung} gives, after absorbing the constant factors into \(C_n>0\),
\[
    \mathcal B(g,1,\ldots,1)
    =
    C_n\int_0^1\int_0^1
    \frac{t^n}{\sqrt{1-t^2}}
    g\bigl((1-t^2s^2)^{(n-1)/2}\bigr)
    (1-s^2)^{(n-2)/2}\,\diff s\,\diff t.
\]
With \(x=ts\) and Fubini's theorem, this becomes
\[
    \mathcal B(g,1,\ldots,1)
    =
    C_n\int_0^1
    g\bigl((1-x^2)^{(n-1)/2}\bigr)
    \cdot
    \left(
        \int_x^1
        \frac{t(t^2-x^2)^{(n-2)/2}}
        {\sqrt{1-t^2}}\,\diff t
    \right)\diff x.
\]
Let \(\mathrm B\) denote the beta function. By the substitution
\(u=(t^2-x^2)/(1-x^2)\) and Euler's beta integral \cite[Eq.~(5.12.1)]{DLMF-Beta}, we obtain
\begin{equation} \label{beta}
    \int_x^1
    \frac{t(t^2-x^2)^{(n-2)/2}}
    {\sqrt{1-t^2}}\,\diff t
    =
    \frac12
    \mathrm B\left(\frac n2,\frac12\right)
    (1-x^2)^{(n-1)/2}.
\end{equation}
Together with the change of variables
\(y=(1-x^2)^{(n-1)/2}\), this gives
\[
    \mathcal B(g,1,\ldots,1)
    =
    C_n\int_0^1
    g(y)\frac{y^{2/(n-1)}}
    {\sqrt{1-y^{2/(n-1)}}}\,\diff y
\]
with \(C_n>0\). Applying the preceding formula to \(g\) and to the constant function \(1\) gives
\[
    \frac{\mathcal B(g,1,\ldots,1)}
         {\mathcal B(1,\ldots,1)}
    =
    \frac{\displaystyle
        \int_0^1
        g(y)\frac{y^{2/(n-1)}}
        {\sqrt{1-y^{2/(n-1)}}}\,\diff y}
    {\displaystyle
        \int_0^1
        \frac{y^{2/(n-1)}}
        {\sqrt{1-y^{2/(n-1)}}}\,\diff y}.
\]
Since the two integrals are \(I(b_g)\) and \(I(b_0)\), respectively, \eqref{eq:B-I-quotient} follows.
\end{proof}

\begin{lemma}[Injectivity of \(g\mapsto\alpha_g\)]
\label{lem:alpha-injektiv}
The linear operator
\[
    C([0,1])\to C([0,1]),
    \qquad
    g\mapsto\alpha_g,
\]
is injective.
\end{lemma}

\begin{proof}
Let \(g\in C([0,1])\) with \(\alpha_g\equiv0\). The substitutions
\(z=t^2\) and \(u=zs^2\) in the one-dimensional representation
\eqref{eq:alpha-definition} of \(\alpha_g\) yield, for \(0<z\leq1\),
\begin{equation}\label{eq:alpha-Abel}
    z^{(n-1)/2}\alpha_g(\sqrt z)
    =
    C_n
    \int_0^z
    (z-u)^{(n-2)/2}u^{-1/2}
    g\bigl((1-u)^{(n-1)/2}\bigr)\,\diff u
\end{equation}
with a constant \(C_n>0\).

For \(0<u\leq1\), set
\[
    h(u)
    =
    u^{-1/2}g\bigl((1-u)^{(n-1)/2}\bigr),
\]
and set \(h(0)=0\). Since \(g\) is bounded, we have \(h\in L^1([0,1])\). Let \(\Gamma\)
denote the gamma function and let \(\mathcal I^\beta\) be the
normalized Riemann--Liouville fractional integral
\[
    (\mathcal I^\beta h)(z)
    =
    \frac1{\Gamma(\beta)}
    \int_0^z
    (z-u)^{\beta-1}h(u)\,\diff u,
    \qquad \beta>0.
\]
For background, see \cite[Section~2]{SamkoKilbasMarichev1993}.
Up to a positive constant, the right-hand side of
\eqref{eq:alpha-Abel} is \(\mathcal I^{n/2}h\). Since
\(\alpha_g\equiv0\), we therefore have
\[
    \mathcal I^{n/2}h=0.
\]

Choose an integer \(k>n/2\). By the semigroup property of the
Riemann--Liouville fractional integrals
\cite[Theorem~2.5]{SamkoKilbasMarichev1993},
\[
    \mathcal I^kh
    =
    \mathcal I^{k-n/2}\mathcal I^{n/2}h
    =
    0
\]
almost everywhere on \((0,1)\). For integer \(k\), the formula for
repeated integration gives
\[
    \frac{\diff^k}{\diff z^k}\mathcal I^kh
    =
    h
\]
in the sense of distributions on \((0,1)\). Hence, \(h=0\) almost
everywhere.

Since
\[
    g\bigl((1-u)^{(n-1)/2}\bigr)
    =
    u^{1/2}h(u),
    \qquad 0<u\leq1,
\]
the continuous function
\(u\mapsto g\bigl((1-u)^{(n-1)/2}\bigr)\) vanishes almost everywhere
and therefore everywhere on \([0,1]\). Since
\(u\mapsto(1-u)^{(n-1)/2}\) maps \([0,1]\) onto itself, it follows
that \(g\equiv0\). Thus, \(g\mapsto\alpha_g\) is injective.
\end{proof}

With these preparations, we can now prove the analytic inequality, characterize its equality cases, and then use it to derive the Petty inequality for convex bodies of revolution.

\begin{lemma}[Analytic inequality] \label{lem:analytisch}
For every continuous function \(g\colon[0,1]\to[0,\infty)\),
\begin{equation} \label{eq:analytic-inequality}
    \frac{T(b_g)}{T(b_0)}
    \geq
    \left(\frac{I(b_g)}{I(b_0)}\right)^{n-1}.
\end{equation}
Equality holds if and only if \(g\) is constant.
\end{lemma}

\begin{proof}
Since \(g\geq0\), we have \(\alpha_g\geq0\). Applying Hölder's inequality to the integral representation in Lemma~\ref{lem:B-faktorisierung}, with weight
\[
    \frac{t^n}{\sqrt{1-t^2}}
\]
and exponents \(n-1\) and \((n-1)/(n-2)\), yields
\begin{equation}\label{eq:Holder}
    \mathcal B(g,1,\ldots,1)^{n-1}
    \leq
    \mathcal B(g,\ldots,g)
    \mathcal B(1,\ldots,1)^{n-2}.
\end{equation}
Since \(\mathcal B(1,\ldots,1)>0\), Lemma~\ref{lem:determinantenvergleich} and \eqref{eq:B-I-quotient} yield
\[
\begin{aligned}
    \frac{T(b_g)}{T(b_0)}
    &\geq
    \frac{\mathcal B(g,\ldots,g)}
         {\mathcal B(1,\ldots,1)}\\
    &\geq
    \left(
        \frac{\mathcal B(g,1,\ldots,1)}
             {\mathcal B(1,\ldots,1)}
    \right)^{n-1}\\
    &=
    \left(\frac{I(b_g)}{I(b_0)}\right)^{n-1}.
\end{aligned}
\]

If equality holds, then equality must in particular hold in \eqref{eq:Holder}. Since the weight in the Hölder inequality is positive almost everywhere and \(\alpha_1\) is a positive constant, the equality case in Hölder's inequality, see \cite[Problem~7.2.15]{Heil2019}, gives
\[
    \alpha_g(t)=c\,\alpha_1(t)
\]
for almost every \(t\in(0,1)\) and some \(c\geq0\). By continuity, this identity holds on all of \([0,1]\). Linearity and Lemma~\ref{lem:alpha-injektiv} therefore yield \(\alpha_{g-c}=0\), and hence \(g\equiv c\).

Conversely, if \(g\equiv c\) with \(c\geq0\), then \(b_g=cb_0\). By the homogeneity of \(T\) and \(I\),
\[
    T(b_g)=c^{n-1}T(b_0),
    \qquad
    I(b_g)=cI(b_0),
\]
and hence equality holds in the stated inequality.
\end{proof}

\section{Smooth Convex Bodies of Revolution}

We first replace \(K\) by the normalized Blaschke sum \(2^{-1/(n-1)}(K\#(-K))\), which is origin-symmetric, has the same projection body as \(K\), and does
not decrease its volume. After this reduction and a suitable linear
normalization, we use the functionals \(I\) and \(T\) to translate the
analytic inequality \eqref{eq:analytic-inequality} into the corresponding geometric
inequality for \(V_n(K)\) and \(V_n(\Pi K)\).

To express \(V_n(\Pi K)\) in terms of \(T\), we use the following volume formula for zonoids.

\begin{lemma}[Volume formula for zonoids]
\label{lem:zonoid-volume}
Let \((X,\mu)\) be a finite measure space and let
\(q\colon X\to\mathbb R^n\) be integrable. If the zonoid \(Z\subseteq\mathbb R^n\)
is defined by
\[
    h_Z(u)
    =
    \int_X |\langle u,q(x)\rangle|\,\diff\mu(x),
    \qquad u\in\mathbb S^{n-1},
\]
then
\[
    V_n(Z)
    =
    \frac{2^n}{n!}
    \int_{X^n}
    \left|
        \det\bigl(
            q(x_1),\ldots,q(x_n)
        \bigr)
    \right|
    \,\diff\mu(x_1)\cdots\diff\mu(x_n).
\]
\end{lemma}

\begin{proof}
Set \(\widehat q(x)=q(x)/\|q(x)\|\) whenever \(q(x)\neq0\), and
\(\widehat q(x)=e_n\) otherwise. Define a finite even Borel measure
\(\nu\) on \(\mathbb S^{n-1}\) by
\[
    \int_{\mathbb S^{n-1}} f(v)\,\diff\nu(v)
    =
    \frac12
    \int_X
    \|q(x)\|
    \bigl(f(\widehat q(x))+f(-\widehat q(x))\bigr)
    \,\diff\mu(x)
\]
for \(f\in C(\mathbb S^{n-1})\). Then
\[
    h_Z(u)
    =
    \int_{\mathbb S^{n-1}}
    |\langle u,v\rangle|\,\diff\nu(v),
\]
so \(\nu\) is a generating measure of \(Z\). The volume formula for
zonoids \cite[Theorem~5.3.2, in particular (5.82)]{Schneider2014}
therefore gives
\[
    V_n(Z)
    =
    \frac{2^n}{n!}
    \int_{(\mathbb S^{n-1})^n}
    |\det(v_1,\ldots,v_n)|
    \,\diff\nu(v_1)\cdots\diff\nu(v_n).
\]
By the definition of \(\nu\) and the invariance of the absolute
determinant under sign changes, the last integral equals
\[
    \int_{X^n}
    \left|
        \det\bigl(q(x_1),\ldots,q(x_n)\bigr)
    \right|
    \,\diff\mu(x_1)\cdots\diff\mu(x_n),
\]
which proves the claim.
\end{proof}

\begin{theorem}[Smooth case for convex bodies of revolution]
\label{thm:petty-revolution}
Let \(n\geq3\), and let \(K\subseteq\mathbb R^n\) be a convex body of revolution of class \(C^\infty_+\). Then
\begin{equation}\label{eq:Petty}
    V_n(\Pi K)
    \geq
    \frac{\kappa_{n-1}^{\,n}}{\kappa_n^{\,n-2}}
    V_n(K)^{n-1}.
\end{equation}
Equality holds if and only if \(K\) is an ellipsoid.
\end{theorem}

\begin{proof}
\emph{Blaschke symmetrization.}
Consider the normalized Blaschke symmetral \(K^{\mathrm B}\) defined
in \eqref{eq:blaschke-symmetral}. Then
\[
    S_{K^{\mathrm B}}
    =
    \frac12(S_K+S_{-K}).
\]
In particular, \(K^{\mathrm B}\) is a body of revolution and origin-symmetric. Since \(K\in C^\infty_+\), the surface area measures \(S_K\) and \(S_{-K}\) have smooth positive densities. Hence, the same is true for \(S_{K^{\mathrm B}}=\frac12(S_K+S_{-K})\), and the regularity theory of the Minkowski problem \cite[Theorem~1]{ChengYau1976} implies that \(K^{\mathrm B}\) is of class \(C^\infty_+\).

Moreover, Cauchy's projection formula and the evenness of the kernel \(v\mapsto|\langle u,v\rangle|\) give
\[
    h_{\Pi K^{\mathrm B}}(u)
    =
    \frac14\int_{\mathbb S^{n-1}}
    |\langle u,v\rangle|\,\diff(S_K+S_{-K})(v)
    =
    h_{\Pi K}(u).
\]
Thus, \(\Pi K^{\mathrm B}=\Pi K\). The Kneser--Süss inequality \cite[Theorem~8.2.3]{Schneider2014} yields
\begin{equation} \label{eq:kneser-suess}
    V_n(K^{\mathrm B})^{(n-1)/n}
    =
    \frac12V_n(K\#(-K))^{(n-1)/n}
    \geq
    V_n(K)^{(n-1)/n}.
\end{equation}
It therefore suffices to consider the origin-symmetric case.

\emph{Normalization and boundary profile.}
The quotient \(V_n(\Pi K)/V_n(K)^{n-1}\) is affine invariant. Since the projection of \(K\) onto \(e_n^\perp\) is a Euclidean ball, a scaling in \(e_n^\perp\) normalizes it to \(B^{n-1}\). Then
\[
    K
    =
    \{(x,z):x\in B^{n-1},\ |z|\leq f(|x|)\},
\]
where \(f\) is concave and \(f(1)=0\). For \(0<y<1\), set
\[
    b_K(y)
    =
    -f'(y^{1/(n-1)}),
    \qquad
    g(y)
    =
    \frac{b_K(y)}{b_0(y)}.
\]
Then \(b_g=b_K\). To apply Lemma~\ref{lem:analytisch}, it remains to verify that \(g\) extends to a positive continuous function on \([0,1]\). The assumption \(K\in C^\infty_+\) is used here to control the boundary profile at the axis and at the equator.

At the axis, smooth rotational symmetry gives \(f'(0)=0\), while
positive curvature gives \(f''(0)<0\). Setting
\(a_0=-\frac12f''(0)>0\), we obtain
\[
    f(r)=f(0)-a_0r^2+O(r^4),
    \qquad
    -f'(r)\sim2a_0r
    \quad\text{as }r\downarrow0.
\]

At the equator, the boundary of \(K\) can be parametrized, for
\(|z|\) sufficiently small, as
\[
    (\theta,z)\mapsto(\psi(z)\theta,z),
    \qquad
    \theta\in\mathbb S^{n-2},
\]
where \(\psi(0)=1\). Origin symmetry gives \(\psi(-z)=\psi(z)\), while
positive curvature gives \(\psi''(0)<0\). Hence, with
\(a_1=-\frac12\psi''(0)>0\),
\[
    \psi(z)=1-a_1z^2+O(z^4),
    \qquad
    \psi'(z)\sim-2a_1z
    \quad\text{as }z\to0.
\]
Since \(r=\psi(z)\), \(z=f(r)\), and \(f'(r)=1/\psi'(z)\), it follows
that
\[
    -f'(r)
    \sim
    \frac{1}{\sqrt{2a_1}}\frac1{\sqrt{1-r^2}}
    \quad\text{as }r\uparrow1.
\]
Thus, \(g\) extends to a positive continuous function on \([0,1]\)
with \(g(0)=2a_0\) and \(g(1)=1/\sqrt{2a_1}\).

\emph{Volume of \(K\).}
Since \(f'(r)\) becomes singular as \(r\uparrow1\), we first integrate by parts on \([0,R]\) with \(R<1\). Using \(y=r^{n-1}\), we obtain
\[
\begin{aligned}
    I(b_g)
    &=
    (n-1)\lim_{R\uparrow1}
    \int_0^R[-f'(r)]r^{n-1}\,\diff r\\
    &=
    (n-1)\lim_{R\uparrow1}
    \left(
        -f(R)R^{n-1}
        +(n-1)\int_0^R f(r)r^{n-2}\,\diff r
    \right).
\end{aligned}
\]
Since \(f(1)=0\), letting \(R\uparrow1\) yields
\[
    I(b_g)
    =
    (n-1)^2\int_0^1f(r)r^{n-2}\,\diff r.
\]
Therefore,
\[
    V_n(K)
    =
    2(n-1)\kappa_{n-1}
    \int_0^1f(r)r^{n-2}\,\diff r
    =
    \frac{2\kappa_{n-1}}{n-1}I(b_g).
\]

\emph{Volume of \(\Pi K\).}
Set \(F(x)=f(|x|)\) and
\[
    p(x)=-\nabla F(x),
    \qquad x\in\operatorname{int}B^{n-1}.
\]
Under the graph parametrization \(x\mapsto(x,F(x))\), the outer unit normal satisfies
\[
    \nu(x,F(x))\,\diff\mathcal H^{n-1}(x,F(x))
    =
    (p(x),1)\,\diff x.
\]

The equator \(\partial K\cap e_n^\perp\) has \(\mathcal H^{n-1}\)-measure zero. Hence, Cauchy's projection formula \cite[Eq.~(5.80)]{Schneider2014} and origin symmetry give
\[
    h_{\Pi K}(u)
    =
    \int_{B^{n-1}}
    |\langle u,(p(x),1)\rangle|\,\diff x.
\]
Thus, \(\Pi K\) is the zonoid generated by the map
\(x\mapsto(p(x),1)\) and Lebesgue measure on \(B^{n-1}\). The generating map is integrable, since
\[
    |p(r\theta)|
    =
    -f'(r)
    =
    O\bigl((1-r^2)^{-1/2}\bigr)
    \quad \text{as } r\uparrow1,
\]
and
\[
    \int_0^1
    \frac{r^{n-2}}{\sqrt{1-r^2}}\,\diff r
    <\infty.
\]
Hence \(p\) is integrable on \(B^{n-1}\), and therefore
\(x\mapsto(p(x),1)\) is integrable on \(B^{n-1}\). Applying Lemma~\ref{lem:zonoid-volume} with \(X=B^{n-1}\) and \(q(x)=(p(x),1)\), we obtain
\[
\begin{aligned}
    V_n(\Pi K)
    &=
    \frac{2^n}{n!}
    \int_{(B^{n-1})^n}
    \left|
        \det\bigl(
            (p(x_1),1),\ldots,(p(x_n),1)
        \bigr)
    \right|
    \,\diff x_1\cdots\diff x_n.
\end{aligned}
\]

Writing \(x=r\theta\) and setting \(y=r^{n-1}\), we have
\[
    p(y^{1/(n-1)}\theta)
    =
    b_g(y)\theta,
    \qquad
    \diff x
    =
    \kappa_{n-1}\,\diff y\,
    \diff\overline\sigma_{n-2}(\theta).
\]
The change of variables in each of the \(n\) integration variables therefore gives
\[
    V_n(\Pi K)
    =
    \frac{2^n\kappa_{n-1}^{\,n}}{n!}\,T(b_g).
\]

For the unit ball, \(f_0(r)=\sqrt{1-r^2}\) corresponds to the function \(b_0\) via \(b_0(y) = -f_0'(y^{1/(n-1)})\). Thus,
\begin{equation}\label{eq:volume-quotients}
    \frac{V_n(K)}{V_n(B^n)}
    =
    \frac{I(b_g)}{I(b_0)},
    \qquad
    \frac{V_n(\Pi K)}{V_n(\Pi B^n)}
    =
    \frac{T(b_g)}{T(b_0)}.
\end{equation}

Lemma~\ref{lem:analytisch} and \eqref{eq:volume-quotients} imply
\[
    \frac{V_n(\Pi K)}{V_n(\Pi B^n)}
    \geq
    \left(
        \frac{V_n(K)}{V_n(B^n)}
    \right)^{n-1}.
\]
Since \(\Pi B^n=\kappa_{n-1}B^n\), this is precisely \eqref{eq:Petty}.

\emph{Equality case.}
In the case of equality in the origin-symmetric normalized setting, Lemma~\ref{lem:analytisch} first gives \(g\equiv c\) with \(c>0\). Hence
\[
    -f'(r)
    =
    c\frac{r}{\sqrt{1-r^2}}.
\]
Using \(f(1)=0\), integration gives
\[
    f(r)
    =
    c\sqrt{1-r^2},
\]
and \(K\) is an ellipsoid of revolution.

For the original body, equality must also hold in the Kneser--Süss inequality \eqref{eq:kneser-suess}. Thus, \(K\) and \(-K\) are homothetic and, since they have the same volume, translates of one another. Consequently, \(S_{-K}=S_K\) and hence \(S_{K^{\mathrm B}}=S_K\). By uniqueness in the Minkowski problem \cite[Theorem~8.1.1]{Schneider2014}, \(K^{\mathrm B}\) and \(K\) are translates of one another, and therefore \(K\) is also an ellipsoid.

The converse follows from affine invariance.
\end{proof}

\section{Profile Measures and the Nonsmooth Equality Case}

The inequality \eqref{eq:Petty} also follows for nonsmooth bodies \(K\) by approximation. To treat the equality cases without smoothness assumptions, however, we pass from the description of \(K\) and \(B^n\) in terms of \(f\) and \(b_K\) to a suitable profile measure \(\rho_K\). This measure also captures possible cylindrical boundary parts parallel to the axis of revolution, which are not represented by the almost everywhere defined derivative \(f'\), and allows us to extend the relevant analytic functionals to the nonsmooth case.

In what follows, \(\delta_a\) denotes the Dirac measure at \(a\), \(|\rho|\) the total variation of a finite signed Borel measure \(\rho\), and \(g\rho\) the measure defined by \(\diff(g\rho)=g\,\diff\rho\).

\begin{definition}
\label{def:profilmass}
Let \(K\subseteq\mathbb R^n\) be an origin-symmetric convex body of revolution with \(\operatorname{proj}_{e_n^\perp}K=B^{n-1}\). Then \(K\) has the representation
\[
    K
    =
    \{(x,z):x\in B^{n-1},\ |z|\leq f(|x|)\},
\]
where \(f:[0,1]\to[0,\infty)\) is concave and nonincreasing. For almost every \(y\in(0,1)\), set
\[
    b_K(y)
    =
    -f'\bigl(y^{1/(n-1)}\bigr)
\]
and define the profile measure of \(K\) by
\begin{equation*}
    \diff\rho_K(y)
    =
    b_K(y)y^{1/(n-1)}\,\diff y
    +(n-1)f(1)\,\delta_1.
\end{equation*}
The reference measure of the unit ball is
\begin{equation*}
    \diff\rho_0(y)
    =
    b_0(y)y^{1/(n-1)}\,\diff y
    =
    \frac{y^{2/(n-1)}}
         {\sqrt{1-y^{2/(n-1)}}}\,\diff y.
\end{equation*}
\end{definition}

Concavity and monotonicity ensure continuity of \(f\) at \(0\), while the closedness of \(K\) gives \(\lim_{r\uparrow1}f(r)=f(1)\). Thus, \(f\) is continuous and concave on \([0,1]\), and therefore absolutely continuous. Using \(y=r^{n-1}\) and the definition of \(\rho_K\), we obtain
\[
\begin{aligned}
    \int_{(0,1]}y^{-1}\,\diff\rho_K(y)
    &=
    (n-1)\int_0^1[-f'(r)]\,\diff r
    +(n-1)f(1)\\
    &=
    (n-1)f(0)<\infty.
\end{aligned}
\]
The following definition therefore applies, in particular, to all
profile measures. The transform \(A_\rho\) introduced below is the
measure-theoretic extension of \(g\mapsto\alpha_g\). Its formula is
obtained by rewriting \(\alpha_g\) in terms of the measure
\(\rho=g\rho_0\).

\begin{definition}[Measure extension of \(\mathcal B\)]
\label{def:B-mass}
Let \(\rho\) be a finite signed Borel measure on \([0,1]\) with
\[
    \rho(\{0\})=0
    \quad \text{and} \quad
    \int_{(0,1]}y^{-1}\,\diff|\rho|(y)<\infty.
\]
For \(0<t<1\), set
\begin{equation}\label{eq:A-rho}
\begin{aligned}
    A_\rho(t)
    =
    \frac{2\kappa_{n-2}}{(n-1)t^{n-1}}
        \int_{\left[(1-t^2)^{(n-1)/2},\,1\right]}
    \frac{
        \bigl(y^{2/(n-1)}-1+t^2\bigr)^{(n-2)/2}
    }{y}
    \,\diff\rho(y).
\end{aligned}
\end{equation}
For such measures \(\rho_1,\ldots,\rho_{n-1}\), define
\begin{equation}\label{eq:B-mass}
\begin{aligned}
    \mathfrak B(\rho_1,\ldots,\rho_{n-1})
    =
     c_n
    \int_0^1
    \frac{t^n}{\sqrt{1-t^2}}\,
    A_{\rho_1}(t)\cdots A_{\rho_{n-1}}(t)\,\diff t,
\end{aligned}
\end{equation}
where \(c_n\) is the constant from Lemma~\ref{lem:B-faktorisierung}.
\end{definition}

The integrals in \eqref{eq:B-mass} converge absolutely, since \eqref{eq:A-rho} gives
\[
    |A_\rho(t)|
    \leq
    \frac{2\kappa_{n-2}}{(n-1)t}
    \int_{(0,1]}y^{-1}\,\diff|\rho|(y).
\]
If \(g\in C([0,1])\), then the change of variables
\(y=(1-t^2s^2)^{(n-1)/2}\) in \eqref{eq:A-rho} shows that
\(A_{g\rho_0}(t)=\alpha_g(t)\) for \(0<t<1\). Thus, Lemma~\ref{lem:B-faktorisierung} immediately yields
\begin{equation}\label{eq:B-erweiterung}
    \mathfrak B(g_1\rho_0,\ldots,g_{n-1}\rho_0)
    =
    \mathcal B(g_1,\ldots,g_{n-1}).
\end{equation}

We now pass from the smooth setting to profile measures and establish the comparison needed to treat equality without regularity assumptions.

\begin{proposition}[Profile measure estimate]
\label{prop:profilmass}
Let \(K\subseteq\mathbb R^n\) be an origin-symmetric convex body of revolution with
\(
    \operatorname{proj}_{e_n^\perp}K=B^{n-1}.
\)
Then
\begin{equation}\label{eq:profilmass-kette}
    \frac{V_n(\Pi K)}{V_n(\Pi B^n)}
    \geq
    \frac{\mathfrak B(\rho_K,\ldots,\rho_K)}
         {\mathfrak B(\rho_0,\ldots,\rho_0)}
    \geq
    \left(
        \frac{\rho_K([0,1])}{\rho_0([0,1])}
    \right)^{n-1}
    =
    \left(
        \frac{V_n(K)}{V_n(B^n)}
    \right)^{n-1}.
\end{equation}
Equality in the second inequality holds if and only if \(\rho_K=c\rho_0\) for some constant \(c\geq0\).
\end{proposition}

\begin{proof}
\emph{Volume and approximation.} For every function \(\varphi\in C^1([0,1])\), the substitution \(y=r^{n-1}\) and integration by parts yield
\begin{equation}\label{eq:profilmass-partielle-integration}
\begin{aligned}
    \int_{[0,1]}\varphi(y)\,\diff\rho_K(y)
    =
    (n-1)^2\int_0^1 f(r)
    \Bigl[
        r^{n-2}\varphi(r^{n-1})
        +r^{2n-3}\varphi'(r^{n-1})
    \Bigr]\,\diff r.
\end{aligned}
\end{equation}
For \(\varphi=1\), it follows from
\eqref{eq:profilmass-partielle-integration} that
\[
    \rho_K([0,1])
    =
    (n-1)^2\int_0^1f(r)r^{n-2}\,\diff r.
\]
On the other hand, by the representation of \(K\) in terms of \(f\) and
polar coordinates in \(\mathbb R^{n-1}\),
\[
    V_n(K)
    =
    2(n-1)\kappa_{n-1}
    \int_0^1f(r)r^{n-2}\,\diff r.
\]
Hence
\begin{equation}\label{eq:volumen-profilmass}
    V_n(K)
    =
    \frac{2\kappa_{n-1}}{n-1}\rho_K([0,1]).
\end{equation}

Let \(\widetilde K_j, j \in \mathbb N,\) be origin-symmetric bodies of revolution of class \(C^\infty_+\) with \(\widetilde K_j\to K\) in the Hausdorff metric. Write \(\operatorname{proj}_{e_n^\perp}\widetilde K_j=r_jB^{n-1}\). Since orthogonal projections are continuous with respect to the Hausdorff metric, the convergence \(\widetilde K_j\to K\) implies \(r_j\to1\). After the scaling
\[
    K_j
    =
    \{(r_j^{-1}x,z):(x,z)\in\widetilde K_j\}
\]
we still have \(K_j\to K\), and at the same time
\(
    \operatorname{proj}_{e_n^\perp}K_j=B^{n-1}.
\)

If \(f_j\) are the corresponding boundary profiles, then \(f_j\to f\) locally uniformly on \([0,1)\) and, by the uniform boundedness of the \(f_j\), also in \(L^1([0,1])\). Thus, \eqref{eq:profilmass-partielle-integration} gives
\[
    \int_{[0,1]}\varphi\,\diff\rho_{K_j}
    \longrightarrow
    \int_{[0,1]}\varphi\,\diff\rho_K
    \qquad
    \text{for all }\varphi\in C^1([0,1]).
\]
Taking \(\varphi=1\), we obtain, in particular, convergence and hence uniform boundedness of the total masses. Since \(C^1([0,1])\) is dense in \(C([0,1])\), the preceding convergence extends to all continuous test functions. Thus, we have the weak convergence
\begin{equation}\label{eq:profilmass-konvergenz}
    \rho_{K_j}\to \rho_K.
\end{equation}

\emph{Determinant estimate.} For \(0<y<1\), set \(g_j(y)=b_{K_j}(y)/b_0(y)\). As in the smooth case, \(g_j\) extends to a positive continuous function on \([0,1]\), \(b_{g_j}=b_{K_j}\), and \(\rho_{K_j}=g_j\rho_0\). Therefore, \eqref{eq:volume-quotients}, Lemma~\ref{lem:determinantenvergleich}, and \eqref{eq:B-erweiterung} yield
\begin{equation}\label{eq:glatter-determinantenvergleich-mass}
    \frac{V_n(\Pi K_j)}{V_n(\Pi B^n)}
    \geq
    \frac{\mathfrak B(\rho_{K_j},\ldots,\rho_{K_j})}
         {\mathfrak B(\rho_0,\ldots,\rho_0)}.
\end{equation}
For \(0<t<1\), set, with \(a_+=\max\{a,0\}\),
\[
    k_t(0)=0,
    \qquad
    k_t(y)
    =
    \frac{
        \bigl(y^{2/(n-1)}-1+t^2\bigr)_+^{(n-2)/2}
    }{y},
    \qquad 0<y\leq1.
\]
Then \(k_t\) is a bounded continuous function on \([0,1]\), and
\[
    A_\rho(t)
    =
    \frac{2\kappa_{n-2}}{(n-1)t^{n-1}}
    \int_{[0,1]}k_t(y)\,\diff\rho(y).
\]
It therefore follows from \eqref{eq:profilmass-konvergenz} that \(A_{\rho_{K_j}}(t)\to A_{\rho_K}(t)\). Fatou's lemma gives
\[
    \mathfrak B(\rho_K,\ldots,\rho_K)
    \leq
    \liminf_{j\to\infty}
    \mathfrak B(\rho_{K_j},\ldots,\rho_{K_j}).
\]
Since the projection body operator and volume are continuous with respect to the Hausdorff metric, it follows from \eqref{eq:glatter-determinantenvergleich-mass} that the first inequality in \eqref{eq:profilmass-kette} holds.

\emph{Hölder's inequality.} 
Let \(\rho\) be a nonnegative measure satisfying the assumptions of Definition~\ref{def:B-mass}. For \(y\in(0,1]\), set \(x(y)=\sqrt{1-y^{2/(n-1)}}\). Since \(A_{\rho_0}=\alpha_1\) is a positive constant, substituting \eqref{eq:A-rho} into \eqref{eq:B-mass} and applying Tonelli's theorem yield, with a constant \(C_n>0\),
\[
\begin{aligned}
    \mathfrak B(\rho,\rho_0,\ldots,\rho_0)
    =
    C_n\int_{(0,1]}\frac1y
    \left(
        \int_{x(y)}^1
        \frac{
            t\bigl(t^2-x(y)^2\bigr)^{(n-2)/2}
        }{\sqrt{1-t^2}}
        \,\diff t
    \right)
    \diff\rho(y).
\end{aligned}
\]
By the beta identity \eqref{beta}, the inner integral equals
\[
    \frac12
    \mathrm B\left(\frac n2,\frac12\right)
    \bigl(1-x(y)^2\bigr)^{(n-1)/2}
    =
    \frac12
    \mathrm B\left(\frac n2,\frac12\right)y.
\]
Consequently, with another constant \(C_n'>0\),
\[
    \mathfrak B(\rho,\rho_0,\ldots,\rho_0)
    =
    C_n'\rho([0,1]).
\]
This computation also includes a possible atom at \(y=1\), since
\(x(1)=0\) and the same beta identity applies there. Applying the
preceding formula to \(\rho\) and to \(\rho_0\) gives
\begin{equation}\label{eq:B-mass-quotient}
    \frac{\mathfrak B(\rho,\rho_0,\ldots,\rho_0)}
         {\mathfrak B(\rho_0,\ldots,\rho_0)}
    =
    \frac{\rho([0,1])}{\rho_0([0,1])}.
\end{equation}
Hölder's inequality with respect to the measure \(\frac{t^n}{\sqrt{1-t^2}}\,\diff t\) gives, exactly as in the proof of Lemma~\ref{lem:analytisch},
\[
\begin{aligned}
    \mathfrak B(\rho_K,\rho_0,\ldots,\rho_0)^{n-1}
    \leq
    \mathfrak B(\rho_K,\ldots,\rho_K)
    \mathfrak B(\rho_0,\ldots,\rho_0)^{n-2}.
\end{aligned}
\]
Together with \eqref{eq:B-mass-quotient}, it follows that
\[
    \frac{\mathfrak B(\rho_K,\ldots,\rho_K)}
         {\mathfrak B(\rho_0,\ldots,\rho_0)}
    \geq
    \left(
        \frac{\rho_K([0,1])}{\rho_0([0,1])}
    \right)^{n-1}.
\]
Using \eqref{eq:volumen-profilmass}, this gives the second inequality in \eqref{eq:profilmass-kette}.

\emph{Equality case.} Since \(A_{\rho_0}=\alpha_1\) is a positive constant, equality in Hölder's inequality, see \cite[Problem~7.2.15]{Heil2019}, gives \(A_{\rho_K}(t)=cA_{\rho_0}(t)\) for almost every \(t\in(0,1)\) and some \(c\geq0\). For every compact interval \([a,b]\subseteq(0,1)\), the kernel \((t,y)\mapsto k_t(y)\) defined above is jointly continuous and bounded on \([a,b]\times[0,1]\). Thus, \(A_\rho\), for every measure \(\rho\) in the above class, is continuous on \((0,1)\) by dominated convergence. Consequently,
\[
    A_{\rho_K}(t)=cA_{\rho_0}(t)
\]
for all \(t\in(0,1)\).

It remains to show that
\[
    A_{\rho_K}=cA_{\rho_0}
    \qquad\Longrightarrow\qquad
    \rho_K=c\rho_0,
\]
that is, to extend the injectivity from Lemma~\ref{lem:alpha-injektiv} to profile measures. Set \(\nu=\rho_K-c\rho_0\), so that \(A_\nu=0\), and define on \((0,1]\) the finite signed measure \(\widehat\nu\) by \(\diff\widehat\nu(y)=y^{-1}\diff\nu(y)\). Let \(\mu_\nu\) denote the pushforward of \(\widehat\nu\) under the map \(y\mapsto1-y^{2/(n-1)}\). For a finite signed measure \(\mu\) and \(\gamma>1\), we use the corresponding measure-valued analogue of the Riemann--Liouville integral and set
\[
    (\mathcal I^\gamma\mu)(z)
    =
    \frac1{\Gamma(\gamma)}
    \int_{[0,z]}(z-u)^{\gamma-1}\,\diff\mu(u).
\]
With \(z=t^2\), \eqref{eq:A-rho} becomes
\begin{equation*}
    z^{(n-1)/2}A_\nu(\sqrt z)
    =
    \frac{2\kappa_{n-2}}{n-1}
    \int_{[0,z]}
    (z-u)^{(n-2)/2}\,\diff\mu_\nu(u).
\end{equation*}
Up to a positive constant, the right-hand side is the Riemann--Liouville integral \(\mathcal I^{n/2}\mu_\nu\). Since \(A_\nu=0\), it follows that \(\mathcal I^{n/2}\mu_\nu=0\) on \((0,1)\).

Choose an integer \(k>n/2+1\), and set
\[
    \alpha=\frac n2,
    \qquad
    \beta=k-\frac n2>1.
\]
The semigroup property remains valid for the finite signed measure
\(\mu_\nu\). Indeed, Fubini's theorem gives
\[
\begin{aligned}
    \mathcal I^\beta\mathcal I^\alpha\mu_\nu(z)
    &=
    \frac1{\Gamma(\alpha)\Gamma(\beta)}
    \int_{[0,z]}
    \left(
        \int_u^z
        (z-s)^{\beta-1}(s-u)^{\alpha-1}\,\diff s
    \right)
    \diff\mu_\nu(u)\\
    &=
    \frac1{\Gamma(k)}
    \int_{[0,z]}(z-u)^{k-1}\,\diff\mu_\nu(u)
    =
    \mathcal I^k\mu_\nu(z).
\end{aligned}
\]
Fubini's theorem is applicable because the same calculation with \(\diff|\mu_\nu|\) gives a finite integral. Consequently, \(\mathcal I^k\mu_\nu=0\) on \((0,1)\). For integer \(k\), taking the distributional derivative of order \(k-1\) of the integral representation of \(\mathcal I^k\mu_\nu\) gives
\[
    \frac{\diff^{k-1}}{\diff z^{k-1}}
    \mathcal I^k\mu_\nu(z)
    =
    \mu_\nu([0,z])
\]
for almost every \(z\in(0,1)\). The distributional derivative of the
function \(z\mapsto\mu_\nu([0,z])\) is
\(\mu_\nu|_{(0,1)}\). Consequently,
\[
    \frac{\diff^k}{\diff z^k}
    \mathcal I^k\mu_\nu
    =
    \mu_\nu|_{(0,1)}
\]
in the sense of distributions on \((0,1)\). Since
\(\mathcal I^k\mu_\nu=0\) on \((0,1)\), it follows that
\(\mu_\nu\) vanishes on \((0,1)\). Moreover, by construction, \(\mu_\nu(\{1\})=0\), since the map \(y\mapsto1-y^{2/(n-1)}\) sends \((0,1]\) into \([0,1)\). Thus, at most \(\mu_\nu=a\delta_0\) remains for some \(a\in\mathbb R\). In this case, however,
\[
    \mathcal I^{n/2}\mu_\nu(z)
    =
    \frac{a}{\Gamma(n/2)}z^{n/2-1}.
\]
Since \(\mathcal I^{n/2}\mu_\nu=0\) on \((0,1)\), it follows that \(a=0\). Thus, \(\mu_\nu=0\). Since the map \(y\mapsto1-y^{2/(n-1)}\) is injective on \((0,1]\), it follows that \(\widehat\nu=0\). Together with \(\nu(\{0\})=0\), this yields \(\nu=0\), that is, \(\rho_K=c\rho_0\).

Conversely, if \(\rho_K=c\rho_0\) for some \(c\geq0\), then the homogeneity of \(\mathfrak B\) gives
\[
    \frac{\mathfrak B(\rho_K,\ldots,\rho_K)}
         {\mathfrak B(\rho_0,\ldots,\rho_0)}
    =
    c^{n-1}
    =
    \left(
        \frac{\rho_K([0,1])}{\rho_0([0,1])}
    \right)^{n-1},
\]
so equality holds in the second inequality of \eqref{eq:profilmass-kette}.
\end{proof}

\section{Proof of the Main Result}

We now combine the results of the preceding sections. The inequality follows by smooth approximation, while the characterization of the equality case is based on the profile measure statement proved above.

\begin{theorem}[Petty conjecture for convex bodies of revolution]
\label{thm:petty-revolution-allgemein}
Let \(n\geq3\), and let \(K\subseteq\mathbb R^n\) be a convex body of revolution. Then
\begin{equation} \label{eq:petty-general}
    V_n(\Pi K)
    \geq
    \frac{\kappa_{n-1}^{\,n}}{\kappa_n^{\,n-2}}
    V_n(K)^{n-1}.
\end{equation}
Equality holds if and only if \(K\) is an ellipsoid.
\end{theorem}

\begin{proof}
By smoothing the support function while preserving rotational symmetry and adding a vanishing Euclidean ball, we choose convex bodies of revolution \(K_j\) of class \(C^\infty_+\) with \(K_j\to K\) in the Hausdorff metric. By Theorem~\ref{thm:petty-revolution},
\[
    V_n(\Pi K_j)
    \geq
    \frac{\kappa_{n-1}^{\,n}}{\kappa_n^{\,n-2}}
    V_n(K_j)^{n-1}.
\]
The continuity of volume and of the projection body operator yields the stated inequality for \(K\) upon passing to the limit.

We now consider the equality case. Let \(K^{\mathrm B}\) be the
Blaschke symmetral from \eqref{eq:blaschke-symmetral}, which is again
a body of revolution. As in the first step of the proof of
Theorem~\ref{thm:petty-revolution}, we have \(\Pi K^{\mathrm B}=\Pi K\)
and \(V_n(K^{\mathrm B})\geq V_n(K)\). Both statements hold without
any regularity assumption. Comparing the Petty inequality for
\(K^{\mathrm B}\) with equality for \(K\) gives
\(V_n(K^{\mathrm B})=V_n(K)\) and equality for \(K^{\mathrm B}\).

After a scaling in \(e_n^\perp\), we may assume that
\(
    \operatorname{proj}_{e_n^\perp}K^{\mathrm B}=B^{n-1}.
\)
Let \(f:[0,1]\to[0,\infty)\) denote the boundary profile of the normalized body \(K^{\mathrm B}\) from Definition~\ref{def:profilmass}. The two outer terms in the chain \eqref{eq:profilmass-kette} from Proposition~\ref{prop:profilmass} then agree, so equality also holds in the second inequality. Consequently, \(\rho_{K^{\mathrm B}}=c\rho_0\) for some \(c>0\). Since \(\rho_0\) has no atom at \(1\), we first obtain \(f(1)=0\). If we compare the absolutely continuous parts of the identity
\(\rho_{K^{\mathrm B}}=c\rho_0\), we obtain
\[
    b_{K^{\mathrm B}}(y)y^{1/(n-1)}
    =
    c\frac{y^{2/(n-1)}}{\sqrt{1-y^{2/(n-1)}}}
\]
for almost every \(y\in(0,1)\). Thus, with the substitution
\(r=y^{1/(n-1)}\), we obtain
\[
    -f'(r)
    =
    c\frac{r}{\sqrt{1-r^2}}
\]
for almost every \(r\in(0,1)\). Hence, \(f(r)=c\sqrt{1-r^2}\), and the normalized image of \(K^{\mathrm B}\) is an ellipsoid of revolution. Consequently, the original body \(K^{\mathrm B}\) is an ellipsoid.

The equality \(V_n(K^{\mathrm B})=V_n(K)\), established before
normalization, gives equality in \eqref{eq:kneser-suess}. The argument
at the end of the proof of Theorem~\ref{thm:petty-revolution} applies
without any regularity assumption and shows that \(K\) is a translate
of \(K^{\mathrm B}\). Thus, \(K\) is an ellipsoid.

Conversely, every ellipsoid attains equality in \eqref{eq:petty-general}, by affine invariance of the Petty quotient.
\end{proof}

\begin{corollary}[Affine extension]
\label{cor:petty-affine-extension}
Let \(n\geq3\), and let \(K\subseteq\mathbb R^n\) be a convex body whose
Blaschke symmetral \(K^{\mathrm B}\) is affinely equivalent to a body of
revolution. Then \eqref{eq:petty-general} holds, with equality if and
only if \(K\) is an ellipsoid.

The assumption on \(K^{\mathrm B}\) is equivalent to requiring
\(\Pi K\) to be linearly equivalent to a body of revolution.
\end{corollary}

\begin{proof}
For \(A\in\operatorname{GL}(n)\), the transformation formula
\[
    \Pi(AK)=|\det A|A^{-T}\Pi K
\]
and translation invariance imply that the Petty quotient is affine
invariant. Thus, Theorem~\ref{thm:petty-revolution-allgemein} applies
to a suitable affine image of \(K^{\mathrm B}\). Together with
\(\Pi K^{\mathrm B}=\Pi K\) and \eqref{eq:kneser-suess}, this gives
\[
    \frac{V_n(\Pi K)}{V_n(K)^{n-1}}
    \geq
    \frac{V_n(\Pi K^{\mathrm B})}
         {V_n(K^{\mathrm B})^{n-1}}
    \geq
    \frac{\kappa_{n-1}^{\,n}}{\kappa_n^{\,n-2}}.
\]
Equality forces \(K^{\mathrm B}\) to be an ellipsoid and equality in
\eqref{eq:kneser-suess}. The argument at the end of the proof of
Theorem~\ref{thm:petty-revolution} then shows that \(K\) is a translate
of \(K^{\mathrm B}\), hence an ellipsoid. Conversely, ellipsoids attain
equality by affine invariance.

It remains to prove the equivalence of the two assumptions. One
implication follows from the transformation formula, the identity
\(\Pi K^{\mathrm B}=\Pi K\), and the fact that the projection body of
a body of revolution is again a body of revolution. Conversely,
suppose that \(T\Pi K\) is a body of revolution for some
\(T\in\operatorname{GL}(n)\), and set
\(A=|\det T|^{1/(n-1)}T^{-T}\). Then
\(\Pi(AK^{\mathrm B})=T\Pi K\). Since this body is origin-symmetric,
its axis passes through the origin. For every rotation \(R\) fixing
this axis pointwise,
\[
    \Pi(RAK^{\mathrm B})
    =R\Pi(AK^{\mathrm B})
    =\Pi(AK^{\mathrm B}).
\]
The injectivity of the projection body operator on origin-symmetric
convex bodies \cite[Theorem~3.3.6]{Gardner2006} yields
\(RAK^{\mathrm B}=AK^{\mathrm B}\). Hence, \(AK^{\mathrm B}\) is a
body of revolution.
\end{proof}

\section*{Acknowledgements}
The author would like to thank Franz Schuster for bringing Petty's
conjecture to his attention, Fabian Mussnig for carefully reading the
manuscript and for his valuable comments and suggestions, and Felix
Dorrek for an interesting exchange concerning the problem.

\section*{Note on the Use of Generative AI}

ChatGPT 5.6 Sol Max was used as a supporting tool both in the search for proof approaches and in their development and the linguistic presentation of the paper. All arguments were selected, developed further, critically examined, and brought into their present form by the author. The author bears full responsibility for the content.

\printbibliography

@online{ChenFengLiXiXu2026,
  author  = {Chen, Shibing and Feng, Yibin and Li, Yuanyuan and
             Xi, Dongmeng and Xu, Lei},
  title   = {{Petty}'s conjectured projection inequality in
             dimension three},
  date    = {2026-08-07},
  note    = {Preprint, MathSciDoc: 2608.23001},
  url     = {https://archive.ymsc.tsinghua.edu.cn/pacm_download/741/12780-pettyinequality.pdf},
  urldate = {2026-08-25}
}

@article{ChengYau1976,
  author  = {Cheng, Shiu-Yuen and Yau, Shing-Tung},
  title   = {On the regularity of the solution of the
             $n$-dimensional {Minkowski} problem},
  journal = {Communications on Pure and Applied Mathematics},
  volume  = {29},
  number  = {5},
  pages   = {495--516},
  year    = {1976},
  doi     = {10.1002/cpa.3160290504}
}

@online{DLMF-Beta,
  author       = {{NIST Digital Library of Mathematical Functions}},
  title        = {Beta Function},
  organization = {National Institute of Standards and Technology},
  note         = {Equation (5.12.1)},
  url          = {https://dlmf.nist.gov/5.12},
  urldate      = {2026-08-22}
}

@online{Dorrek2026,
  author  = {Dorrek, Felix},
  title   = {Proving a math conjecture with {AI}, and the shape of
             machine intelligence},
  date    = {2026-08},
  url     = {https://felixdorrek.com/writing/petty-conjecture},
  urldate = {2026-08-24}
}

@book{Federer1969,
  author    = {Federer, Herbert},
  title     = {Geometric Measure Theory},
  series    = {Die Grundlehren der mathematischen Wissenschaften},
  number    = {153},
  location  = {Berlin},
  publisher = {Springer-Verlag},
  year      = {1969}
}

@book{Gardner2006,
  author    = {Gardner, Richard J.},
  title     = {Geometric Tomography},
  edition   = {2},
  series    = {Encyclopedia of Mathematics and its Applications},
  number    = {58},
  location  = {Cambridge},
  publisher = {Cambridge University Press},
  year      = {2006},
  doi       = {10.1017/CBO9781107341029}
}

@online{gardner2025geometrictomographyupdate,
  author  = {Gardner, Richard J.},
  title   = {{Geometric Tomography}, Second Edition:
             Corrections and Update},
  date    = {2025-11-28},
  note    = {Version 2.1},
  url     = {https://faculty.gardner.wwu.edu/Update%20Version%202_1%20changes.pdf},
  urldate = {2026-07-29}
}

@article{GiannopoulosPapadimitrakis1999,
  author  = {Giannopoulos, Apostolos A. and
             Papadimitrakis, Michael},
  title   = {Isotropic surface area measures},
  journal = {Mathematika},
  volume  = {46},
  number  = {1},
  pages   = {1--13},
  year    = {1999},
  doi     = {10.1112/S0025579300007518}
}

@book{Heil2019,
  author    = {Heil, Christopher},
  title     = {Introduction to Real Analysis},
  series    = {Graduate Texts in Mathematics},
  number    = {280},
  location  = {Cham},
  publisher = {Springer},
  year      = {2019},
  doi       = {10.1007/978-3-030-26903-6}
}

@article{ivaki2018localuniqueness,
  author  = {Ivaki, Mohammad N.},
  title   = {A local uniqueness theorem for minimizers of
             {Petty}'s conjectured projection inequality},
  journal = {Mathematika},
  volume  = {64},
  number  = {1},
  pages   = {1--19},
  year    = {2018},
  doi     = {10.1112/S0025579317000444}
}

@article{Lutwak1990Quermassintegrals,
  author  = {Lutwak, Erwin},
  title   = {On quermassintegrals of mixed projection bodies},
  journal = {Geometriae Dedicata},
  volume  = {33},
  number  = {1},
  pages   = {51--58},
  year    = {1990},
  doi     = {10.1007/BF00147600}
}

@incollection{lutwak1990conjecturedprojection,
  author    = {Lutwak, Erwin},
  title     = {On a conjectured projection inequality of {Petty}},
  editor    = {Grinberg, Eric and Quinto, Eric Todd},
  booktitle = {Integral Geometry and Tomography},
  series    = {Contemporary Mathematics},
  number    = {113},
  pages     = {171--182},
  location  = {Providence, RI},
  publisher = {American Mathematical Society},
  year      = {1990},
  doi       = {10.1090/conm/113/1108653}
}

@article{ortegamoreno2021fixedpointsminkowskivaluations,
  author  = {Ortega-Moreno, Oscar and Schuster, Franz E.},
  title   = {Fixed points of {Minkowski} valuations},
  journal = {Advances in Mathematics},
  volume  = {392},
  eid     = {108017},
  year    = {2021},
  doi     = {10.1016/j.aim.2021.108017}
}

@inproceedings{petty1971isoperimetricproblems,
  author    = {Petty, C. M.},
  title     = {Isoperimetric problems},
  editor    = {Kay, David C.},
  booktitle = {Proceedings of the Conference on Convexity and
               Combinatorial Geometry},
  pages     = {26--41},
  location  = {Norman, OK},
  publisher = {Department of Mathematics, University of Oklahoma},
  year      = {1971}
}

@article{saroglou2011volumesprojectionbodies,
  author  = {Saroglou, Christos},
  title   = {Volumes of projection bodies of some classes of
             convex bodies},
  journal = {Mathematika},
  volume  = {57},
  number  = {2},
  pages   = {329--353},
  year    = {2011},
  doi     = {10.1112/S0025579311001860}
}

@article{saroglou2015secondprojectionbody,
  author  = {Saroglou, Christos},
  title   = {On the shape of a convex body with respect to its
             second projection body},
  journal = {Advances in Applied Mathematics},
  volume  = {67},
  pages   = {55--74},
  year    = {2015},
  doi     = {10.1016/j.aam.2015.03.004}
}

@article{SaroglouZvavitch2017,
  author  = {Saroglou, Christos and Zvavitch, Artem},
  title   = {Iterations of the projection body operator and a remark
             on {Petty}'s conjectured projection inequality},
  journal = {Journal of Functional Analysis},
  volume  = {272},
  number  = {2},
  pages   = {613--630},
  year    = {2017},
  doi     = {10.1016/j.jfa.2016.08.015}
}

@article{schneider1987geometricinequalities,
  author  = {Schneider, Rolf},
  title   = {Geometric inequalities for {Poisson} processes of
             convex bodies and cylinders},
  journal = {Results in Mathematics},
  volume  = {11},
  number  = {1--2},
  pages   = {165--185},
  year    = {1987},
  doi     = {10.1007/BF03323266}
}

@book{Schneider2014,
  author    = {Schneider, Rolf},
  title     = {Convex Bodies: The {Brunn--Minkowski} Theory},
  edition   = {2},
  series    = {Encyclopedia of Mathematics and its Applications},
  number    = {151},
  location  = {Cambridge},
  publisher = {Cambridge University Press},
  year      = {2014},
  doi       = {10.1017/CBO9781139003858}
}

@book{SchneiderWeil2008,
  author    = {Schneider, Rolf and Weil, Wolfgang},
  title     = {Stochastic and Integral Geometry},
  series    = {Probability and Its Applications},
  location  = {Berlin},
  publisher = {Springer},
  year      = {2008},
  doi       = {10.1007/978-3-540-78859-1}
}

@article{Weil1971,
  author  = {Weil, Wolfgang},
  title   = {{\"U}ber die Projektionenk{\"o}rper konvexer Polytope},
  journal = {Archiv der Mathematik},
  volume  = {22},
  pages   = {664--672},
  year    = {1971},
  doi     = {10.1007/BF01222633}
}

@book{SamkoKilbasMarichev1993,
  author    = {Samko, Stefan G. and Kilbas, Anatoly A. and
               Marichev, Oleg I.},
  title     = {Fractional Integrals and Derivatives:
               Theory and Applications},
  location  = {Yverdon},
  publisher = {Gordon and Breach Science Publishers},
  year      = {1993},
  isbn      = {978-2-88124-864-1}
}

@article{BaranyHugReitznerSchneider2017,
  author  = {B{\'a}r{\'a}ny, Imre and Hug, Daniel and
             Reitzner, Matthias and Schneider, Rolf},
  title   = {Random points in halfspheres},
  journal = {Random Structures \& Algorithms},
  volume  = {50},
  number  = {1},
  pages   = {3--22},
  year    = {2017},
  doi     = {10.1002/rsa.20644}
}

\end{document}